\documentclass[12pt,reqno]{amsart}

\usepackage{amsmath}
\usepackage{amssymb}
\usepackage{amsfonts}
\usepackage{setspace}
\usepackage{mathrsfs}
\usepackage{version}

\DeclareMathAlphabet{\curly}{U}{refs}{m}{n}  

\newtheorem{theorem}{Theorem}[section]
\newtheorem{lemma}[theorem]{Lemma}
\newtheorem{proposition}[theorem]{Proposition}
\newtheorem{corollary}[theorem]{Corollary}

\theoremstyle{definition}

\renewcommand{\leq}{\leqslant}
\renewcommand{\geq}{\geqslant}

\def\R{\mathbb{R}}
\def\C{\mathbb{C}}
\def\Z{\mathbb{Z}}
\def\E{\mathbb{E}}
\def\P{\mathbb{P}}
\def\Q{\mathbb{Q}}
\def\N{\mathbb{N}}

\def\X{\mathbf{X}}
\def\Y{\mathbf{Y}}
\def\ZZ{\mathbf{Z}}

\def\eps{\varepsilon}

\makeatletter
\renewcommand{\pmod}[1]{\allowbreak\mkernmu({\operator@font mod}\,\,#1)}
\makeatother

\newcommand{\be}{\begin{equation}}
\newcommand{\ee}{\end{equation}}

\newcommand{\ssum}[1]{\sum_{\substack{#1}}}  
\renewcommand{\(}{\left(}
\renewcommand{\)}{\right)}

\newcommand{\order}{\asymp}      
\newcommand{\fl}[1]{{\ensuremath{\left\lfloor {#1} \right\rfloor}}}

\renewcommand{\le}{\leqslant}
\renewcommand{\ge}{\geqslant}

\newcommand{\cS}{\mathcal{S}}  
\newcommand{\cA}{\mathcal{A}}
\newcommand{\sL}{\mathscr{L}}

\newcommand{\bthet}{\boldsymbol{\theta}}
\newcommand{\T}{\mathbb{T}}
 
\numberwithin{equation}{section}

\begin{document}

\title{Invariable generation of the symmetric group}



\author{Sean Eberhard}
\address{Mathematical Institute\\
Radcliffe Observatory Quarter\\
Woodstock Road\\
Oxford OX2 6GG\\
England}
\email{sean.eberhard@maths.ox.ac.uk}

\author{Kevin Ford}
\address{Department of Mathematics, 1409 West Green Street, University
of Illinois at Urbana-Champaign, Urbana, IL 61801, USA}
\email{ford@math.uiuc.edu}

\author{Ben Green}
\address{Mathematical Institute\\
Radcliffe Observatory Quarter\\
Woodstock Road\\
Oxford OX2 6GG\\
England}
\email{ben.green@maths.ox.ac.uk}

\thanks{BG is supported by ERC Starting Grant 279438 ``Approximate Algebraic Structure'' and a Simons Investigator Grant.}
\thanks{KF is supported by National Science Foundation grants DMS-1201442 and DMS-1501982.}

\begin{abstract}
We say that permutations $\pi_1,\dots, \pi_r \in \cS_n$ \emph{invariably generate} $\cS_n$ if, no matter how one chooses conjugates $\pi'_1,\dots,\pi'_r$ of these permutations, $\pi'_1,\dots,\pi'_r$ generate $\cS_n$. We show that if $\pi_1,\pi_2,\pi_3$ are chosen randomly from $\cS_n$ then, with probability tending to 1 as $n \rightarrow \infty$, they do not invariably generate $\cS_n$. By contrast it was shown recently by Pemantle, Peres and Rivin that four random elements do invariably generate $\cS_n$ with probability bounded away from zero. We include a proof of this statement which, while sharing many features with their argument, is short and completely combinatorial.
\end{abstract}

\maketitle

\section{Introduction}

Albeit by Dixon's theorem~\cite{dixon1} two random elements $\pi_1,\pi_2$ of the symmetric group $\cS_n$ generate at least the whole alternating group $\cA_n$ with high probability\footnote{We adopt the convention that for a sequence of events $E_n$ in finite probability spaces depending on some parameter $n$, ``$E_n$ occurs with high probability'' means with $\P(E_n)\to 1$ as $n\to\infty$.} as $n\to\infty$, it is less clear how large the group generated by $\pi'_1,\pi'_2$ must be when $\pi'_1$ and $\pi'_2$ are allowed to be arbitrary conjugates of $\pi_1$ and $\pi_2$. Following Dixon~\cite{dixon2} we say that a list $\pi_1,\dots,\pi_r\in\cS_n$ has a property $P$ \emph{invariably} if $\pi_1',\dots,\pi_r'$ has property $P$ whenever $\pi_i'$ is conjugate to $\pi_i$ for every $i$. How many random elements of $\cS_n$ must we take before we expect them to invariably generate $\cS_n$?

Several authors~\cite{davenportsmith,dixon2,kowalskizywina,lp93,musser,PPR} have already considered this question, owing to its connection with computational Galois theory. To briefly explain this connection, suppose we are given a polynomial $f\in\Z[x]$ of degree $n$ with no repeated factors. Information about the Galois group can be gained by reducing $f$ modulo various primes $p$ and factorizing the reduced polynomial $\bar{f}$ over $\Z/p\Z$. By classical Galois theory, if $\bar{f}$ has irreducible factors of degrees $n_1,\dots,n_r$ then the Galois group $G$ of $f$ over $\Q$ has an element with cycle lengths $n_1,\dots,n_r$. Moreover by Frobenius's density theorem, if $G=\cS_n$ then the frequency with which a given cycle type arises is equal to the proportion of elements in $\cS_n$ with that cycle type. Thus if we suspect that $G=\cS_n$ then the number of times we expect to have to iterate this procedure before proving that $G=\cS_n$ is controlled by the expected number of random elements required to invariably generate $\cS_n$.

{\L}uczak and Pyber~\cite{lp93} were the first to prove the existence of a constant $C$ such that $C$ random permutations $\pi_1,\dots,\pi_C\in\cS_n$ invariably generate $\cS_n$ with probability bounded away from zero. Their method does not directly yield a reasonable value of $C$, but recently Pemantle, Peres, and Rivin~\cite{PPR} proved that we may take $C=4$.

\begin{theorem}[Pemantle--Peres--Rivin~\cite{PPR}]\label{fourgen}
If $\pi_1,\pi_2,\pi_3,\pi_4\in\cS_n$ are chosen uniformly at random then the probability that $\pi_1,\pi_2,\pi_3,\pi_4$ invariably generate $\cS_n$ is bounded away from zero.
\end{theorem}

Incidentally, Pemantle, Peres, and Rivin only prove that $\pi_1,\pi_2,\pi_3,\pi_4$ invariably generate a transitive subgroup of $\cS_n$, but it is little more work to prove the theorem as stated above. We give a somewhat simplified proof of this theorem in Section~\ref{sec:fourgen}. Our main contribution however is the lower bound $C>3$, which will be proved in Section~\ref{sec:threegen}. Thus $C$ can be taken as small as $4$, but no smaller.

\begin{theorem}\label{threegen}
If $\pi_1,\pi_2,\pi_3\in\cS_n$ are chosen uniformly at random then the probability that $\pi_1,\pi_2,\pi_3$ invariably generate a transitive subgroup \textup{(}or, in particular, all of $\cS_n$\textup{)} tends to zero as $n\to\infty$. Equivalently, with probability tending to $1$ there is a positive integer $k<n$ such that $\pi_1,\pi_2,\pi_3$ each have a fixed set of size $k$.
\end{theorem}

As in our recent paper~\cite{EFG1}, our main tool is the following model for the small-cycle structure of a random permutation; see for example Arratia and Tavar\'e~\cite{AT92}.

\begin{lemma}\label{Poisson}
Let $\X=(X_1,X_2,\dots)$ be a sequence of independent Poisson random variables, where $X_j$ has parameter $1/j$. 
If $c_j$ is the number of cycles of length $j$ in a random permutation $\pi\in\cS_n$, then if $k$ is fixed and $n\to\infty$ the distribution of $(c_1,\dots,c_k)$ converges to that of $(X_1,\dots,X_k)$.
\end{lemma}

The set of fixed-set sizes of a random permutation is thus modeled by the random sumset 
\begin{equation}\label{LX}
\sL(\X) = \left\{\sum_{j\geq 1} jx_j: 0\leq x_j\leq X_j \right\}.
\end{equation}
Thus unsurprisingly the main task in proving Theorem~\ref{fourgen} is to show that
\begin{equation}\label{4-}
\P \big( \sL(\X)\cap\sL(\X')\cap\sL(\X'')\cap\sL(\X''')=\{0\}\big) > 0,
\end{equation}
where $\X',\X'',\X'''$ are independent copies of $\X$. Similarly, Theorem~\ref{threegen} follows almost immediately from
\begin{equation}\label{3-x}
\P \big( \sL(\X)\cap\sL(\X')\cap\sL(\X'') = \{0\} \big) = 0.
\end{equation}
Ultimately, these assertions come down to the inequalities $\log 2 < \frac{3}{4}$ and $\frac{2}{3} < \log 2$ respectively, as we shall see in the course of the proofs.

These questions about permutations have analogues in number theory. Our proof of Theorem~\ref{threegen} is modeled after that of the well-known theorem of Maier and Tenenbaum~\cite{MT} on the \emph{propinquity of divisors}: a random integer $n$ (selected from $\{1,\dots, x\}$ for large $x$) has two distinct divisors $d,d'$ with $d<d'\le 2d$ with high probability (as $x\to \infty$).
In particular, we make heavy use of Riesz products, a device closely related to the sums $\sum_{d | n} d^{i \theta}$ that one sees frequently in the propinquity literature. 

The Maier--Tenenbaum theorem itself corresponds more perfectly with the assertion that with high probability a random permutation $\pi$  has, for some $k$ with $0<k<n$, two different fixed sets of size $k$ (a true statement, but not one we establish here).

The number-theoretic analogue of Theorem \ref{fourgen} is a statement of the following kind: if $x$ is large and if we select four random integers $n_1, n_2, n_3, n_4$ independently at random from $\{1,\dots, x\}$ then, with probability bounded below by an absolute constant, any divisors $d_1|n_1, d_2|n_2, d_3|n_3, d_4|n_4$ should have $\max d_i > 2 \min d_i$. 

The number-theoretic analogue of Theorem \ref{threegen} is a statement of the following kind: if $x$ is large and if we select three random integers $n_1, n_2, n_3$ independently at random from $\{1,\dots, x\}$ then, with probability tending to $1$ as $x \rightarrow \infty$, there exist divisors $d_1|n_1, d_2|n_2, d_3|n_3$ with $\max d_i < 2\min d_i$. 

Both of these number-theoretical statements were established nearly 20 years ago by Raouj and Stef \cite{raouj-stef}. In fact, rather more precise statements are established in that paper.  We thank G\'erald Tenenbaum for bringing this paper
to our attention.

The analogous problem of determining the expected number of random elements
required to invariable generatean arbitrary finite group has been considered
in several recent papers \cite{KLS, kowalskizywina, Luc}.

\subsection*{Notation}
Throughout we use standard $O(\cdot)$ and $o(\cdot)$ notation, as well as the Vinogradov notation $X\ll Y$ to mean $X = O(Y)$.

%
%
%
%
%
%

\section{Four generators are enough}\label{sec:fourgen}

The principal result needed for the proof of Theorem~\ref{fourgen} is the following proposition.

\begin{proposition}\label{fourgen-dyadic}
The following is true uniformly for integers $k,n$ with $1\leq k\leq n/2$. If $\pi_1,\pi_2,\pi_3,\pi_4\in\cS_n$ are chosen uniformly at random, then the probability that there is some $\ell\in(k/2,k]$ such that $\pi_1,\pi_2,\pi_3,\pi_4$ each fix a set of size $\ell$ is $O(k^{-c})$ for some $c>0$.
\end{proposition}

We begin with a tool for counting permutations with a given number of cycles of length at most $k$.
By the Poisson model mentioned in the Introduction (Lemma \ref{Poisson}), if $k$ is fixed and $n\to \infty$,
this statistic has distribution approaching that of $X_1+\cdots+X_k$, a Poisson variable with parameter
$h_k=1+\frac12+\cdots + \frac1{k}$.  The next result tells us that the distribution is still approximately 
Poisson uniformly over all choices of parameters $k$ and $n$.

\begin{lemma}\label{single_set}
Let $n,k,\ell$ be integers with $n\ge k\ge 1$ and $\ell\ge 0$.  Select $\pi\in \cS_n$ at random.  Then
\[
 \P (\pi \text{ has exactly } \ell \text{ cycles with length} \le k) \le \frac{e}{k} 
 \frac{(1+\log k)^\ell}{\ell!} \( 1  + \frac{\ell}{1+\log k} \).
\]
In particular if $\ell \ll \log k$ then this is $O\((1 + \log k)^{\ell}/k \ell!\)$, while if $\ell \gg \log k$ then this is $O\((1+\log k)^{\ell-1}/k(\ell-1)!\)$.
\end{lemma}

\begin{proof}
Denote by $\cS_n(k,\ell)$ the set of $\pi\in \cS_n$ containing exactly $\ell$ cycles of length at most $k$. Evidently
\[
  n |\cS_n(k,\ell)| = \sum_{\pi \in \cS_n(k,\ell) } \ssum{\sigma|\pi \\ \sigma \text{ a cycle}} |\sigma|.
\]
Here the inner sum is over cycles $\sigma$ which are factors of (i.e., contained in) $\pi$, and $|\sigma|$ denotes the length of $\sigma$.
Write $\pi=\sigma \pi'$, and observe that $\pi'$ has either $\ell-1$ or $\ell$ cycles of length at most $k$, depending on whether $|\sigma|\leq k$ or not. 
Thus $\pi'\in S_{n-|\sigma|}(k,m)$, where $m=\ell-1$ or $m=\ell$,
so
\begin{align*}
  n |\cS_n(k,\ell)|
  &\le \sum_{j=1}^n \sum_{m=\ell-1}^\ell \sum_{\pi'\in \cS_{n-j}(k,m)} \; \ssum{\sigma\in\cS_n, |\sigma|=j \\ \sigma\text{ a cycle}} j\\
  &= \sum_{j=1}^n \sum_{m=\ell-1}^\ell \sum_{\pi'\in \cS_{n-j}(k,m)} \frac{n!}{(n-j)!}.
\end{align*}
Now rearrange the sum according the cycle type $(c_1,\ldots,c_{n})$ of
the permutation $\pi'$, i.e., $\pi'$ has $c_i$ cycles of length $i$ for $1\le i\le n$, and $c_1 + 2c_2 + \dots + nc_n = n -j$ if $\pi' \in \cS_{n-j}$. The well known \emph{Cauchy formula} states that the number of $\pi' \in \cS_{n-j}$ with a given cycle type is $(n- j)!/\prod_i c_i! i^{c_i}$.
It follows that
\begin{align*}
  n |\cS_n(k,\ell)| 
  & \leq n! \sum_{j = 1}^n \sum_{\substack{c_1,\dots, c_n \geq 0 \\ c_1 + 2c_2 + \dots + nc_n = n - j \\ c_1 + \dots + c_k \in \{\ell-1, \ell\} }} \frac{1}{\prod_i c_i! i^{c_i}} \\ 
  & \le n!  \sum_{\substack{c_1,\dots, c_n \geq 0 \\ c_1 + \dots + c_k \in \{\ell-1, \ell\} }} \frac{1}{\prod_i c_i! i^{c_i}}\\
  & = n!  \(\sum_{\substack{c_1,\dots, c_k \geq 0 \\ c_1 + \dots + c_k = \ell - 1 }} \frac{1}{\prod_i c_i! i^{c_i}} + \sum_{\substack{c_1,\dots, c_k \geq 0 \\ c_1 + \dots + c_k = \ell }} \frac{1}{\prod_i c_i! i^{c_i}} \) \sum_{c_{k+1},\dots,c_n\geq 0} \frac1{\prod_{i=k+1}^n c_i! i^{c_i}}\\
  &= n! \( \frac{h_k^{\ell-1}}{(\ell-1)!} + 
 \frac{h_k^\ell}{\ell!}\) \prod_{k<i\le n} e^{1/i},
\end{align*}
where in the last line we used the multinomial theorem. The claimed bound now follows using the inequalities $h_k \leq 1+\log k$ and
\[
 \ssum{k<i\le n} \frac{1}{i} = h_n - h_k \le \log n - \log k + 1. \qedhere
\]
\end{proof}

In our paper \cite{EFG1} we showed that the probability of a random permutation $\pi \in \cS_n$ fixing some set of size $k$ is $k^{-\delta + o(1)}$, where $\delta = 1 - \frac{1 + \log \log 2}{\log 2} \approx 0.086$. As noted in that paper, the main contribution to this estimate comes from rather exceptional permutations with an unexpectedly large number, $\approx \log k/\log 2$, of cycles of length $\leq k$. By contrast a typical permutation has $\approx \log k$ cycles of length $\leq k$. By restricting to this ``quenched'' regime\footnote{The terminology is from \cite{PPR} and apparently comes from statistical physics.} we can establish a much stronger bound.

\begin{lemma}\label{normalfix}
Suppose that $k,n$ are integers with $1\leq k\leq n/2$, $0<\eps \le 1/2$, and choose $\pi\in\cS_n$ uniformly at random. Then the probability that $\pi$ fixes a set of size $k$ and has at most $(1+\eps)\log k$ cycles of length at most $k$ is at most 
$O(k^{\log 2 - 1 + 2\eps})$.
\end{lemma}

\begin{proof}
Fix $\ell \leq (1+\eps)\log k$ and consider permutations $\pi$ with exactly $\ell$ cycles of length at most $k$. If $\pi$ fixes some set $X$, $|X| = k$, write $\pi_1 = \pi |_X$ and $\pi_2 = \pi |_{[n] \setminus X}$ for the induced permutations on $X$ and its complement. Then $\pi_1$ has $\ell_1$ cycles of length $\leq k$, and $\pi_2$ has $\ell_2$ cycles of length $\leq k$, where $\ell_1 + \ell_2 = \ell$. By Lemma~\ref{single_set} the number of such $\pi$, for a given choice of $X$ and $\ell_1, \ell_2$, is bounded by a constant times
\[ \frac{(1 + \log k)^{\ell_1}}{k \ell_1!} k! \cdot \frac{(1 + \log k)^{\ell_2}}{k \ell_2!} (n - k)!,\] which means that the probability we are interested in is bounded by a constant times
\[
	\sum_{\ell_1+\ell_2=\ell} \frac{1}{k^2} \frac{(1+\log k)^{\ell}}{\ell_1!\ell_2!} = 
	\frac{2^\ell (1+\log k)^{\ell}}{k^2 \ell! }.
\]
By summing over all $\ell\le \ell_0=\lfloor (1+\eps)\log k \rfloor$ we get the bound
\begin{align*}
\frac{1}{k^2}\sum_{\ell\leq(1+\eps)\log k} \frac{2^{\ell}(1+\log k)^{\ell}}{\ell!} 
 &\ll \frac{1}{k^2} \frac{2^{\ell_0} (1+\log k)^{\ell_0}}{\ell_0!}\\
 &\ll \frac{1}{k^2} (2e/(1+\eps))^{(1+\eps)\log k}\\
 &\ll \frac{1}{k^{1 - \log 2 -2 \eps}}.\qedhere
\end{align*}
\end{proof}

\begin{proof}[Proof of Proposition~\ref{fourgen-dyadic}] Let $\eps >0$ be small and fixed. First we will use Lemma~\ref{single_set} to bound the probability that one of $\pi_1,\pi_2,\pi_3,\pi_4$ has more than $\ell_0=\fl{(1+\eps)\log k}$ cycles of length at most $k$. By that lemma, for each $\ell\geq\ell_0$, the probability that $\pi_1$ has $\ell$ cycles of length at most $k$ is bounded by
\[
  O\(\frac{(1+\log k)^{\ell-1}}{k(\ell-1)!}\),
\]
so the probability that $\pi_1$ has more than $\ell_0$ cycles is bounded by a constant times
\[
  \sum_{\ell > \ell_0} \frac{(1+\log k)^{\ell-1}}{k(\ell-1)!} \ll \frac{(1+\log k)^{\ell_0-1}}{k(\ell_0-1)!} \ll \frac1k \(\frac{e(1+\log k)}{\ell_0-1}\)^{\ell_0-1} \ll \frac1k \(\frac{e}{1+\eps}\)^{(1+\eps)\log k}.
\]
Now by a Taylor expansion of $-1+(1+\eps)\log(e/(1+\eps))$, this is bounded by $O(k^{-\eps^2/3})$ if $\eps\le \frac12$. Thus the probability that one of $\pi_1,\pi_2,\pi_3,\pi_4$ has more than $\ell_0$ cycles of length at most $k$ is also bounded by $O(k^{-\eps^2/3})$.

On the other hand, by Lemma~\ref{normalfix}, for each $\ell\in (k/2,k]$ the probability that $\pi_i$ has at most $(1+\eps)\log k$ cycles of length at most $k$ and fixes a set of size $\ell$ is at most $k^{\log 2 - 1 + 2\eps}$. Thus the probability that $\pi_1,\pi_2,\pi_3,\pi_4$ each have at most $(1+\eps)\log k$ cycles of length at most $k$ and each fix a set of the same size $\ell$ for some $\ell\in(k/2,k]$ is at most $k^{1+4(\log 2 - 1+2\eps)}$. Since $1+4(\log 2 - 1)<0$, we have $1+4(\log 2 - 1 + 2\eps) < 0$ if $\eps$ is small enough ($\eps=1/40$ works), and so the theorem holds with
\[
  c = \min(\eps^2/3, -1-4(\log 2 - 1 + 2\eps)).\qedhere
\]
\end{proof}

An immediate corollary of Proposition \ref{fourgen-dyadic} is obtained by fixing $k$, letting $n\to \infty$, and recalling the Poisson model (Lemma \ref{Poisson}) and the definition~\eqref{LX} of $\sL(\X)$ .

\begin{corollary}\label{smallk}
 For any $k\ge 2$, the probability that $\sL(\X)\cap\sL(\X')\cap\sL(\X'')\cap\sL(\X''')$ contains an
 integer $\ell \in (k/2,k]$ is $O(k^{-c})$, for some $c>0$.
\end{corollary}

\emph{Remark.} If one wished to prove only this, we could substitute Lemma \ref{single_set} with a corresponding bound for $\P(X_1 + \dots + X_k \leq \ell)$, which follows very quickly from the fact that $X_1 + \dots + X_k$ is Poisson with parameter $h_k$. 

\begin{corollary}\label{fourint}
$\sL(\X)\cap\sL(\X')\cap\sL(\X'')\cap\sL(\X''')$ is almost surely finite, and equal to $\{0\}$ with positive probability.
\end{corollary}

\begin{proof}
Let $F_k$ be the event that
\[
  \sL(\X)\cap\sL(\X')\cap\sL(\X'')\cap\sL(\X''')\cap(k,\infty)
\]
is nonempty. By applying Corollary~\ref{smallk} with $k$ replaced by $2^jk$, $j\in \N$, and summing the geometric series, we obtain $\P(F_k) \ll k^{-c}$ for $k\geq 1$. In particular $\P(F_k)\to 0$, so $\P\(\bigcap F_k\) = 0$, so the first part of the corollary holds. For the second part, fix $k_0$ such that $\P(F_{k_0})<1$. Then
\begin{align*}
  \P\(\sL(\X)\cap\sL(\X')\cap\sL(\X'')\cap\sL(\X''') = \{0\}\)
  & \geq \P(X_j = 0~\text{for all $j\leq k_0$, and}~F_{k_0}^c)\\
  & \geq \P(X_j = 0~\text{for all}~j\leq k_0)\,\P(F_{k_0}^c)\\
  & > 0.
\end{align*}
The second inequality here is a simple case of the FKG inequality \cite[Theorem 1.19]{TV}. To see the inequality directly, define $\X^*=(X^*_1,X^*_2,\dots)$ by putting $X^*_j=X_j$ if $j>k_0$ and $X^*_j=0$ if $j\leq k_0$, and let $F_{k_0}^*$ be the event that
\[
	\sL(\X^*)\cap\sL(\X')\cap\sL(\X'')\cap\sL(\X''')\cap(k_0,\infty)
\]
is nonempty. Clearly $F_{k_0}^*$ implies $F_{k_0}$, so
\begin{align*}
\P(X_j = 0~\text{for all $j\leq k_0$, and}~F_{k_0}^c)
  & = \P(X_j = 0~\text{for all}~j\leq k_0,~\text{and}~F_{k_0}^{*c})\\
  & = \P(X_j = 0~\text{for all}~j\leq k_0)\,\P(F_{k_0}^{*c})\\
  &\geq \P(X_j = 0~\text{for all}~j\leq k_0)\,\P(F_{k_0}^c).\qedhere
\end{align*}
\end{proof}


Shortly we will complete the proof of Theorem~\ref{fourgen}. In the proof, we will need a trick to deal with the possibility that $\pi_1,\pi_2,\pi_3, \pi_4 \in \cA_n$. The following lemma is helpful in this regard. It shows that random even and random odd permutations have the same small-cycle structure
as random permutations with unconstrained parity (Lemma \ref{Poisson}).

\begin{lemma}\label{oddperm}
Let $\pi\in\cS_n$ be a random even permutation, and let $c_j(\pi)$ be the number of cycles of length $j$.  Fix $k\in \N$. 
Then as $n\to\infty$ the distribution of $(c_1(\pi),\dots,c_k(\pi))$ converges to that of $(X_1,\dots,X_k)$. The same is true if $\pi$ is a random odd permutation.
\end{lemma}
\begin{proof}
Choose $\pi\in \cS_n$ uniformly at random, and define $\sigma$ by putting $\sigma=1$ if $\pi$ is even and $\sigma=(12)$ if $\pi$ is odd. 
Then $\pi\sigma$ is uniformly distributed over $\cA_n$.  By Lemma \ref{Poisson}, as $n\to \infty$,
the number of cycles in $\pi$ of length at most $2k$ approaches a Poisson distribution with parameter $h_{2k}\le 1+\log 2k$.  
Thus, with high probability (as $n\to\infty$) the total number of points in cycles of $\pi$ of length at most $2k$ is at most $2k\log n$,
so with high probability each of these cycles is disjoint from $(12)$.
That is, the points $1$ and $2$ are both contained in cycles of $\pi$ of length at least $2k+1$ with high probability.
Now consider the probability that 1 and 2 are both contained in the same cycle  and are close together.
For each $\ell\ge 2k+1$, the number of cycles of length $\ell$ containing both 1 and 2, which are a distance $\le k$
from each other, equals $\binom{n-2}{\ell-2} 2k(\ell-2)!$.  Hence, the number of permutations $\pi$ containing such a cycle is  at most
\[
 \sum_{2k+1 \le \ell\le n} 2k (n-2)! \le 2k(n-1)!.
\]
Hence, with high probability, if 1 and 2 are in the same cycle they are a distance 
at least $k+1$ from each other. Thus, with high probability, $c_j(\pi\sigma) = c_j(\pi)$ for each $j\leq k$. 
Similarly $\pi\sigma(12)$ is uniformly distributed over odd permutations, and with high probability $c_j(\pi\sigma(12))=c_j(\pi)$.
\end{proof}

\begin{proof}[Proof of Theorem~\ref{fourgen}]
Let $\pi_1,\pi_2,\pi_3,\pi_4\in\cS_n$ be random permutations with $\pi_1$ odd. Let $E_{n,k}$ be the event that $\pi_1,\pi_2,\pi_3,\pi_4$ each fix a set of size $\ell$ for some $\ell$ in the range $1\leq \ell\leq k$, and let $F_{n,k}$ be the event that $\pi_1,\pi_2,\pi_3,\pi_4$ each fix a set of size $\ell$ for some $\ell$ in the range $k<\ell\leq n/2$. 
By Proposition~\ref{fourgen-dyadic} (and summing a geometric series as in the proof of Corollary~\ref{fourint}) we have $\P(F_{n,k})\ll k^{-c}$ uniformly for $1\leq k\leq n/2$, while by Corollary~\ref{fourint} and Lemma~\ref{oddperm} we have $\lim_{n\to\infty} \P(E_{n,k}) \leq 1-\delta$ for all $k$, for some constant $\delta>0$. 
Fix $k_0$ such that $\P(F_{n,k_0})\leq \delta/3$ for all $n\geq 2k_0$. Then $\P(E_{n,k_0})+\P(F_{n,k_0}) \leq 1 - \delta/3$ for all sufficiently large $n$, so we deduce that with probability bounded away from zero $\pi_1,\pi_2,\pi_3,\pi_4$ do not fix sets of the same size $\ell$ for any $\ell\in[1,n/2]$.

Thus with probability bounded away from zero $\pi_1,\pi_2,\pi_3,\pi_4$ invariably generate a transitive subgroup of $\cS_n$. However by the {\L}uczak--Pyber theorem~\cite{lp93}, $\pi_1$ is, with high probability, not contained in any transitive subgroup smaller than $\cA_n$. Since $\pi_1\not\in\cA_n$, with probability bounded away from zero $\pi_1,\pi_2,\pi_3,\pi_4$ invariably generate $\cS_n$.
\end{proof}

\section{Three generators are not enough}\label{sec:threegen}

Theorem~\ref{threegen} follows immediately from the following more specific proposition.

\begin{proposition}\label{threegen-spec}
For every $\eps>0$ there exists $k_0 = k_0(\eps)$ and $n_0 = n_0(\eps)$ such that if $n\geq n_0$ then with probability at least $1-\eps$ there is some $\ell\leq k_0$ such that $\pi_1,\pi_2,\pi_3$ each fix a set of size $\ell$.
\end{proposition}

Let $\X$ be defined as before, and let $\Y$ and $\ZZ$ be independent copies of $\X$. For $I$ an interval in $\N$ let $$\sL(I,\X) = \left\{\sum_{j\in I} jx_j : 0\leq x_j \leq X_j\right\},$$ and define $\sL(I,\Y)$ and $\sL(I,\ZZ)$ analogously.

\begin{lemma}\label{maxS}
Let $I = \{1,\dots,k\}$ and let $\eps>0$. Then with probability at least $1-\eps$ we have $\sL(I,\X),\sL(I,\Y),\sL(I,\ZZ)\subset [0,3\eps^{-1} k]$.
\end{lemma}
\begin{proof}
Since $\E\sum_{j\in I} j X_j = |I|= k,$ by Markov's inequality we have $\sum_{j\in I} jX_j \leq 3\eps^{-1} k$ with probability at least $1-\eps/3$. Similarly $\sum_{j\in I} j Y_j\leq 3\eps^{-1}k$ and $\sum_{j\in I} jZ_j \leq 3\eps^{-1}k$ each with probability at least $1-\eps/3$, and the lemma follows.
\end{proof}

\begin{lemma}\label{eventE}
Fix $\eps$, $0<\eps<1/2$.  There is a constant $C(\eps)$ so that with probability at least $1-\eps$ we have
\begin{align*}
\sum_{m<j\leq k} X_j &\geq 0.99\log(k/m) - C(\eps),\\
\sum_{m<j\leq k} Y_j &\geq 0.99\log(k/m) - C(\eps),~\text{and}\\
\sum_{m<j\leq k} Z_j &\geq 0.99\log(k/m) - C(\eps).
\end{align*}
for every nonnegative integer $m\leq k$.
\end{lemma}
\begin{proof}
Let $C = C(\eps)$ be a constant whose properties will be specified later. It suffices to show that the first inequality holds for all $m\leq k$ with probability at least $1-\eps/3$. There is nothing to prove if $m\ge e^{-C}k$, so we may suppose $m\leq e^{-C}k$. We may also suppose that $C\ge 1$. 

Let $E$ be the event that $$\sum_{m<j\leq k} X_j \geq 0.99\log(k/m)-1$$ for all $m\leq e^{-C} k$. Suppose $E$ fails, say $$\sum_{m<j\leq k} X_j < 0.99\log(k/m)-1$$ for some $m\leq e^{-C} k$. Writing $m'$ for the smallest power of $2$ with $m'>m$, we thus have $$\sum_{m'<j\leq k} X_j \leq \sum_{m<j\leq k} X_j \leq 0.99 \log(k/m) -1\leq 0.99\log(k/m').$$
Thus
$$1_{E^c} \leq \sum_{\substack{m'\leq 2 e^{-C} k\\\text{dyadic}}} 0.99^{\sum_{m'<j\leq k} X_j-0.99\log(k/m')}.$$
Whenever $P$ is Poisson of parameter $\lambda$ and $a>0$ we have $\E a^P = e^{(a-1)\lambda}$, and the sum $\sum_{m'<j\leq k}X_j$ is Poisson with parameter $\sum_{m'<j\leq k} 1/j = \log(k/m') + O(1)$, so
\begin{align*}
\P(E^c) &\ll \sum_{\substack{m'\leq 2 e^{-C} k\\\text{dyadic}}} \exp\left((0.99-1 - 0.99\log(0.99))\log(k/m')\right)\\
 &\leq \sum_{\substack{m'\leq 2 e^{-C} k\\\text{dyadic}}}(k/m')^{-0.00005}\\
 &\ll e^{-0.00005C}.
\end{align*}
Therefore, $\P(E^c) \le \eps/3$ if $C$ is taken large enough.
\end{proof}

We need a standard estimate for the partial sums of the Fourier series $\sum_{j=1}^\infty \frac{\cos(2\pi j\theta)}{j}=-\log|2\sin(\pi\theta)|$. Denote by $\|x\|$ the distance from $x$ to $\Z$. 

\begin{lemma}\label{trigsum}
\[ \sum_{j \leq m} \frac{\cos(2\pi j \theta)}j = \log \min\(\frac{1}{\|\theta\|},m\) + O(1) \quad
\text{for}~\|\theta\|>0.  \]
\end{lemma}

\begin{proof}
We may assume that $0<\theta\le  \frac12$. Using the bound $\cos(2\pi j\theta) = 1 + O(j^2 \theta^2)$, we get
\[
  \sum_{j \leq \min(m,1/\theta)} \frac{\cos(2\pi j \theta)}j = \log \min(m,1/\theta)+O(1).
\]
This proves the lemma if $\|\theta\| \leq 1/m$. Suppose, then, that $\|\theta\| > 1/m$. Set 
\[
  S_j = \sum_{n=0}^j e^{2\pi i n\theta},
\]
and note that by summing the geometric series we have
\begin{equation}\label{Sjbound}
  S_j =   \frac{e^{2\pi i j\theta}-1}{e^{2\pi i \theta}-1} \ll \frac{1}{\theta}.
\end{equation}
Thus (by ``Abel summation''),
\begin{align*}
 \sum_{1/\theta < j \le m} \frac{\cos(2\pi j \theta)}j &= \Re \sum_{1/\theta < j \le m} \frac{e^{2\pi i j \theta}}j
 =\Re \sum_{1/\theta < j \le m} \frac{S_j-S_{j-1}}{j}\\ &=\Re \sum_{1/\theta < j \le m-1} \frac{S_j}{j(j+1)} + \frac{S_m}{m}
 -\frac{S_{\lceil 1/\theta \rceil-1}}{\lceil 1/\theta \rceil}.
\end{align*}
The latter two terms here are $O(1)$ by the trivial bound $|S_j|\leq j$, while from~\eqref{Sjbound} the sum is bounded by a constant times
\[
  \frac1\theta \sum_{j > 1/\theta} \frac1{j^2} \ll 1.\qedhere
\]
\end{proof}

Let $\T=\R/\Z$ be the unit torus, and denote $e(z)=e^{2\pi i z}$.
Given $I,\X,\Y,\ZZ$ define $F:\T^2 \to\C$ by
$$F(\bthet) = \prod_{j\in I} \(\frac{1+e(j\theta_1)} 2\)^{X_j}\(\frac{1+e(j \theta_2)}2\)^{Y_j}\(\frac{1+e(j(-\theta_1-\theta_2))}2\)^{Z_j}.$$
By expanding the product we see that $\hat{F}:\Z^2\to\C$ is supported on the set
\begin{equation}\label{s-def} S(I,\X,\Y,\ZZ) = \{ (n_1-n_3, n_2-n_3) : n_1\in \sL(I,\X), n_2\in  \sL(I,\Y), n_3\in \sL(I,\ZZ)\}.\end{equation}
Since $\sum_{a\in\Z^2} \hat{F}(a) = F(0) = 1,$ by Cauchy--Schwarz we
have
\[
1=\Bigg( \sum_{a\in\Z^2} \hat{F}(a) \bigg)^2\le \bigg(
\sum_{a:\hat{F}(a)\ne 0} 1 \bigg) \sum_{a\in \Z^2} |\hat{F}(a)|^2 \le
|S(I,\X,\Y,\ZZ)|  \sum_{a\in \Z^2} |\hat{F}(a)|^2.
\]
Applying Parseval, we get
\begin{equation}\label{cauchyparseval}
|S(I,\X,\Y,\ZZ)|  \geq \(\sum_{a\in\Z^2} |\hat{F}(a)|^2\)^{-1} = \(\int_{\T^2} |F(\bthet)|^2 \,d\bthet\)^{-1}.
\end{equation}

\begin{lemma}\label{mainlemma}
Let
\[
  \beta = 1- \frac{2}{3\log 2} - 0.02\approx0.0182,
\]
and let $I=(k^\beta,k]$. Fix $\eps \in (0,1/2)$, and let $E = E(\eps)$ be the event from Lemma~\ref{eventE}. Then both of the bounds
\begin{equation}
\E 1_E |F(\bthet)|^2 \ll_\eps (k\|\theta_1\|^{1/3}  \|\theta_2\|^{1/3}  \|\theta_3\|^{1/3})^{-2.02}\label{t1t2t3}
\end{equation}
and
\begin{equation} \E 1_E |F(\bthet)|^2 \ll_\eps (k\|\theta_i\|^{1/2}  \|\theta_j\|^{1/2})^{-1.3} \quad (\{i,j\} \subset \{1,2,3\})\label{t1t2}
\end{equation}
hold uniformly for $\bthet\in\T^2$, where $\theta_3=-\theta_1-\theta_2$. The expectation is over $\X, \Y, \ZZ$.
\end{lemma}

\begin{proof}
Define, for $i \in \{1,2,3\}$,
\[ k_i = \left\{ \begin{array}{ll}k^{\beta} & \mbox{if $\Vert \theta_i \Vert \geq k^{-\beta}$}; \\ 1/\Vert \theta_i \Vert & \mbox{if $1/k < \Vert \theta_i \Vert < k^{-\beta}$}; \\ k & \mbox{if $\Vert \theta_i \Vert \leq 1/k$}. \end{array}  \right. \]
It is useful to note the (slightly crude) bound
\begin{equation}\label{star-bd} k_i \leq \frac{k^{\beta}}{\Vert \theta_i \Vert^{1 - \beta}},\end{equation} which follows by an analysis of the three cases in the definition of $k_i$. 
If $E$ holds then
$$\sum_{k_1< j\leq k} X_j + \sum_{k_2< j\leq k} Y_j +\sum_{k_3< j\leq k} Z_j \geq 0.99\log(k^3/(k_1k_2k_3)) - C(\eps),$$
so
$$1_E |F(\bthet)|^2 \ll_\eps (k^3/k_1k_2k_3)^{-0.99\log 2} |F(\bthet)|^2 2^{\sum_{k_1<j\leq k} X_j + \sum_{k_2<j\leq k} Y_j + \sum_{k_3<j\leq k} Z_j}.$$
From \eqref{star-bd} and the inequality $3 \times 0.99\log 2 \times (1-\beta) > 2.02$, we deduce that
$$1_E |F(\bthet)|^2 \ll_\eps (k\|\theta_1\|^{1/3} \|\theta_2\|^{1/3} \|\theta_3\|^{1/3})^{-2.02} |F(\bthet)|^2 2^{\sum_{k_1<j\leq k} X_j + \sum_{k_2<j\leq k} Y_j + \sum_{k_3<j\leq k}Z_j}.$$
Thus~\eqref{t1t2t3} will follow if we can prove
\begin{equation}\label{EFtheta1}
\E |F(\bthet)|^2 2^{\sum_{k_1<j\leq k} X_j + \sum_{k_2<j\leq k} Y_j + \sum_{k_3<j\leq k}Z_j} \ll 1.
\end{equation}

Similarly, from \eqref{star-bd} for $i = 1,2$ and the trivial bound $k_3 \leq k$, and using $2\times 0.99\log 2 \times (1-\beta) > 1.3$, we deduce that
$$1_E |F(\bthet)|^2 \ll_\eps (k\Vert\theta_1\Vert^{1/2} \Vert\theta_2\Vert^{1/2})^{-1.3} |F(\bthet)|^2
 2^{\sum_{k_1<j\leq k} X_j + \sum_{k_2<j\leq k} Y_j + \sum_{k_3<j\leq k}Z_j},$$
and similarly for other permutations of the indices $1,2,3$,
 so~\eqref{t1t2}   will also follow from~\eqref{EFtheta1}.

It remains only to prove~\eqref{EFtheta1}. We have a factorization
\begin{align*}
\E|F(\bthet)|^2 &2^{\sum_{k_1<j\leq k} X_j + \sum_{k_2<j\leq k} Y_j + \sum_{k_3<j\leq k}Z_j}\\
 &= \prod_{k^\beta<j\leq k_1} \E \left|\frac{1 + e(j\theta_1)}2\right|^{2X_j} \prod_{k_1<j\leq k} \E \(2 \left|\frac{1 + e(j\theta_1)}2\right|^2\)^{X_j}\\
 &\,\times\prod_{k^\beta<j\leq k_2} \E \left|\frac{1 + e(j\theta_2)}2\right|^{2Y_j} \prod_{k_2<j\leq k} \E \(2 \left|\frac{1 + e(j\theta_2)}2\right|^2\)^{Y_j}\\
 &\,\times\prod_{k^\beta<j\leq k_3} \E \left|\frac{1 + e(j\theta_3)}2\right|^{2Z_j} \prod_{k_3<j\leq k} \E \(2 \left|\frac{1 + e(j\theta_3)}2\right|^2\)^{Z_j}.
\end{align*}
By using again the calculation $\E a^P = e^{(a-1)\lambda}$ for $P$ Poisson with parameter $\lambda$, we get
\begin{align*}
\E|F(\bthet)|^2 &2^{\sum_{k_1<j\leq k} X_j + \sum_{k_2<j\leq k} Y_j + \sum_{k_3<j\leq k}Z_j}\\
 &= \exp\sum_{i=1}^3 \(\sum_{k^\beta<j\leq k_i} \frac1j \(\left|\frac{1 + e(j\theta_i)}2\right|^2-1\) + \sum_{k_i<j\leq k} \frac1j \( 2\left|\frac{1 + e(j\theta_i)}2\right|^2-1\)\)\\
 &= \exp\sum_{i=1}^3 \(\sum_{k^\beta<j\leq k_i} \frac{\cos(2\pi j \theta_i) - 1}{2j} + \sum_{k_i<j\leq k} \frac{\cos(2\pi j\theta_i)}j\)\\
 &= \exp\sum_{i=1}^3 \(\frac12 \log\frac{\min(k_i,1/\|\theta_i\|)}{\min(k^\beta,1/\|\theta_i\|)} - \frac12 \log\frac{k_i}{k^\beta} + \log\frac{\min(k,1/\|\theta_i\|)}{\min(k_i,1/\|\theta_i\|)} + O(1)\)
\end{align*}
by Lemma~\ref{trigsum}. Checking the three cases in the definition of $k_i$ separately, it can be confirmed that this is always $O(1)$.
\end{proof}

\begin{corollary}\label{cor:intexp}
With notation as in Lemma~\ref{mainlemma}, we have
\begin{equation}\label{intbound}
 \int_{\T^2} \E 1_E |F(\bthet)|^2\,d\bthet \ll_\eps k^{-2}.
\end{equation}
\end{corollary}
\begin{proof}
Divide $\T^2$ into three regions $R_1, R_2, R_3$ as follows: 
\begin{align*}
  R_1 &= \{\bthet\in\T^2 : \|\theta_i\| \geq 1/k~\text{for all three}~i\in\{1,2,3\}\},\\
  R_2 &= \{\bthet\in\T^2 : \|\theta_i\|\geq 1/k~\text{for exactly two}~i\in\{1,2,3\}\},\\
  R_3 &= \{\bthet\in\T^2 : \|\theta_i\|\geq 1/k~\text{for at most one}~i\in\{1,2,3\}\}.
\end{align*}
We will bound the integral differently in each region.

Further subdivide $R_1$ according to which of $\|\theta_1\|,\|\theta_2\|,\|\theta_3\|$ is largest. In the subregion $R_1'$ in which say $\|\theta_1\|$ is largest we have $\|\theta_1\|\geq \|\theta_2\|^{1/2}\|\theta_3\|^{1/2}$, so by~\eqref{t1t2t3} we have
\begin{align*}
  \int_{R_1'} \E 1_E |F(\bthet)|^2 \,d\bthet
  &\ll_\eps \int_{R_1} (k \|\theta_2\|^{1/2} \|\theta_3\|^{1/2})^{-2.02} \,d\bthet\\
  &= \(\int_{\|\theta\|\geq 1/k} (k\|\theta\|)^{-1.01} \,d\theta\)^2\\
  &\order k^{-2}.
\end{align*}
We can bound the integral over the other subregions in the same way, so the integral over $R_1$ is indeed $\ll_\eps k^{-2}$.

Similarly, subdivide $R_2$ according to the relative order of $\|\theta_1\|, \|\theta_2\|, \|\theta_3\|$, and focus for the moment on the subregion $R_2'$ in which $\|\theta_1\|\leq \|\theta_2\|\leq \|\theta_3\|$. This implies in particular that $\|\theta_1\|\leq 1/k$ while $\|\theta_2\|\geq 1/k$. Thus by~\eqref{t1t2} with $i=2$ and $j=3$ we have
\[
  \int_{R_2'} \E 1_E |F(\bthet)|^2 \,d\bthet \ll_\eps \int_{R_2'} (k\|\theta_2\|)^{-1.3} \,d\bthet \order k^{-2}.
\]
Again we can bound the integral over the other subregions in the same way, so the integral over $R_2$ is also $\ll_\eps k^{-2}$.

Finally, in the region $R_3$ note that because $\theta_1+\theta_2+\theta_3 = 0$ we must have $\|\theta_i\| < 2/k$ for each $i$. Thus from the trivial bound $|F(\bthet)|\leq 1$ we have
\[
  \int_{R_3} \E 1_E |F(\bthet)|^2 \,d\bthet \leq \int_{R_3} 1 \order k^{-2}.\qedhere
\]
\end{proof}

Recall the definition of $S(I, \X, \Y, \ZZ)$, given in \eqref{s-def}.

\begin{proposition}\label{Sbig}
Let $I=(k^\beta,k]$. There is a constant $c>0$ such that with probability at least $1/2$ we have $S(I,\X,\Y,\ZZ)\subset [-10k,10k]^2$ and $|S(I,\X,\Y,\ZZ)|\geq c k^2$.
\end{proposition}

\begin{proof}
Apply Lemma~\ref{eventE} with $\eps = 0.01$, and let $E$ be the resulting event. By Corollary~\ref{cor:intexp} (and interchanging the order of integration and expectation) we have
\[
  \E 1_E \int_{\T^2} |F(\bthet)|^2\,d\bthet \ll k^{-2}.
\]
Thus by Markov's inequality there is a constant $C$ such that $1_E \int_{\T^2} |F(\bthet)|^2\,d\bthet \leq C k^{-2}$ with probability at least $0.99$. Since $\P(E)\geq 0.99$ we deduce that $\int_{\T^2} |F(\bthet)|^2\,d\bthet \leq C k^{-2}$ with probability at least $0.98$. Applying~\eqref{cauchyparseval}, we have $|S(I,\X,\Y,\ZZ)|\geq C^{-1} k^2$ with probability at least $0.98$.

On the other hand, by Lemma~\ref{maxS} with $\eps=1/3$ we have $S(I,\X,\Y,\ZZ)\subset [-10k,10k]^2$ with probability at least $2/3$, so we must have both $S(I,\X,\Y,\ZZ)\subset [-10k,10k]^2$ and $|S(I,\X,\Y,\ZZ)|\geq C^{-1} k^2$ with probability at least $1 - 1/3 - 0.02 \geq 1/2$.
\end{proof}

\begin{proposition}
Let $I = (k^\beta,60k]$. Then with probability bounded away from zero we can find $(x_j)_{j\in I}, (y_j)_{j\in I}, (z_j)_{j\in I}$ not all zero such that $0\leq x_j\leq X_j$, $0\leq y_j\leq Y_j$, and $0\leq z_j\leq Z_j$ for each $j\in I$ and $$\sum_{j\in I} jx_j = \sum_{j\in I} jy_j = \sum_{j\in I} jz_j.$$
\end{proposition}

\begin{proof}
Let $I'=(k^\beta,k]$. By Proposition~\ref{Sbig}, with probability at least $1/2$ we have $$S(I',\X,\Y,\ZZ)\subset[-10k,10k]^2$$ and $|S(I',\X,\Y,\ZZ)|\gg k^2$. This event depends only on $X_j,Y_j,Z_j$ for $j\leq k$, so independently with probability at least $1/2$ we can find ${j_3}\in(20k,50k]$ such that $Z_{j_3}>0$, as
\begin{equation} \label{Zj-computation}
  \P(Z_j = 0~\text{for all}~j\in(20k,50k]) = \prod_{j\in(20k,50k]} e^{-1/j} \leq 1/2.
\end{equation}
Given such a $j_3$ the set $T$ of pairs of integers $(j_1,j_2)$ such that $10k<j_1,j_2\leq 60k$ and for which
$$j_1 (1,0)+j_2 (0,1)-j_3 (1,1) \in - S(I',\X,\Y,\ZZ)$$ has size $|T|\gg k^2$. In particular there is a set $T_1$ of integers $j_1$ in the range $10k < j_1 \leq 60k$ of size $|T_1|\gg k$ such that for each $j_1\in T_1$ there are $\gg k$ integers $j_2$ in the same range $10k<j_2\leq 60k$ such that $(j_1,j_2)\in T$. Thus by two further computations along the lines of~\eqref{Zj-computation}, independently with probability $\gg 1$ we can find $j_1\in T_1$ such that $X_{j_1}>0$, and then $j_2$ such that $(j_1,j_2)\in T$ and such that $Y_{j_2}>0$.

But then by definition of $S(I',\X,\Y,\ZZ)$ we can find $(x_j)_{j\in I'}, (y_j)_{j\in I'}, (z_j)_{j\in I'}$ such that $0\leq x_j\leq X_j$, $0\leq y_j\leq Y_j$, and $0\leq z_j\leq Z_j$ for all $j\in I'$ and such that
\[
  j_1 + \sum_{j\in I'} jx_j = j_2 + \sum_{j\in I'} jy_j = j_3 + \sum_{j\in I'} jz_j.
\]
Thus the proposition follows from putting $x_{j_1}=y_{j_2}=z_{j_3}=1$, and putting all other $x_j,y_j,z_j$ with $j>k$ equal to $0$.
\end{proof}

\begin{corollary}\label{threeint}
$\sL(\X)\cap\sL(\Y)\cap\sL(\ZZ)$ is almost surely infinite.
\end{corollary}
\begin{proof}
Define $k_1$ to be sufficiently large, and thereafter $k_{i+1} = (60k_i)^{1/\beta}.$ Then the intervals $I_i=(k_i^\beta,60k_i]$ are pairwise disjoint and by the proposition for each the probability that we can find $(x_j)_{j\in I_i}, (y_j)_{j\in I_i}, (z_j)_{j\in I_i}$ not all zero such that $0\leq x_j\leq X_j$, $0\leq y_j\leq Y_j$, and $0\leq z_j\leq Z_j$ for each $j\in I_i$ and $$\sum_{j\in I_i} jx_j = \sum_{j\in I_i} jy_j = \sum_{j\in I_i} jz_j$$
is bounded away from zero. Since these events are independent for different values of $i$ the corollary follows.
\end{proof}

\begin{proof}[Proof of Proposition~\ref{threegen-spec}]
By Corollary~\ref{threeint} there is some $k_0=k_0(\eps)$ such that if $\sL(\X)\cap\sL(\Y)\cap\sL(\ZZ)\cap[1,k_0]$ is nonempty with probability at least $1-\eps/2$. Thus by Lemma~\ref{Poisson} there is some $n_0=n_0(\eps)$ such that if $n\geq n_0$ then with probability at least $1-\eps$ there is some $\ell\leq k_0$ such that $\pi_1,\pi_2,\pi_3$ each fix a set of size $\ell$.
\end{proof}

\medskip

{\bf Acknowledgments.}  The authors thank the referees for carefully reading
the paper and for several helpful suggestions.

\end{document}